\documentclass[11pt]{article}
\renewcommand{\baselinestretch}{1.2}
\newcommand{\dated}{\mbox{} \hfill {\small [{\tt \today}]}} \usepackage{amsmath,amssymb,amsfonts,diagrams}
\newenvironment{keywords}{\noindent\small {\it Keywords\/}:}{\vskip 4pt}
\newenvironment{classification}{\noindent\small 2010 {\it Mathematics Subject
Classification\/}:}{\vskip 12pt}

\newcommand{\comps}{{\mathbb C}}

\newcommand{\tensor}{\otimes}

\newcommand{\Tensor}{\hat{\otimes}}

\newcommand{\cstar}{{C^\ast}}

\newcommand{\id}{{\mathrm{id}}}

\newcommand{\VN}{\operatorname{VN}}
\newcommand{\PM}{\operatorname{PM}}

\usepackage{amsthm,enumerate}
\theoremstyle{plain}
\newtheorem{theorem}{Theorem}[section]
\newtheorem{lemma}[theorem]{Lemma}
\newtheorem{corollary}[theorem]{Corollary}
\newtheorem{proposition}[theorem]{Proposition}
\theoremstyle{definition}
\newtheorem{definition}[theorem]{Definition}
\theoremstyle{remark}
\newtheorem*{remark}{Remark}
\newtheorem*{example}{Example}
\newtheorem*{rems}{Remarks}
\newtheorem*{exs}{Examples}
\newenvironment{remarks}{\begin{rems}\begin{enumerate}}{\end{enumerate}\end{rems}}
\newenvironment{examples}{\begin{exs}\begin{enumerate}}{\end{enumerate}\end{exs}}
\newenvironment{items}{\begin{enumerate}[\rm (i)]}{\end{enumerate}}
\newenvironment{alphitems}{\begin{enumerate}[\rm (a)]}{\end{enumerate}}

\newcommand{\PF}{\mathrm{PF}}

\title{Beurling--Fig\`a-Talamanca--Herz algebras}
\author{\textit{Serap \"Oztop}\thanks{Research supported by the Scientific Projects Coordination Unit of the University of Istanbul under Project Number IRP 11488.} \and \textit{Volker Runde}\thanks{Research supported by NSERC.} \and \textit{Nico Spronk}\thanks{Research supported by NSERC.}}
\date{}
\begin{document}
\maketitle
\begin{abstract}
For a locally compact group $G$ and $p \in (1,\infty)$, we define and study the Beurling--Fig\`a-Talamanca--Herz algebras $A_p(G,\omega)$. For $p=2$ and abelian $G$, these are precisely the Beurling algebras on the dual group $\hat{G}$. For $p =2$ and compact $G$, our approach subsumes an earlier one by H.\ H.\ Lee and E.\ Samei. The key to our approach is not to define Beurling algebras through weights, i.e., possibly unbounded continuous functions, but rather through their inverses, which are bounded continuous functions. We prove that a locally compact group $G$ is amenable if and only if one---and, equivalently, every---Beurling--Fig\`a-Talamanca--Herz algebra $A_p(G,\omega)$ has a bounded approximate identity.
\end{abstract}
\begin{keywords} 
amenable, locally compact group; Beurling algebra; Beurling--Fig\`a--Talamanca--Herz algebra; Beurling--Fourier algebra; bounded approximate identity; inverse of a weight; Leptin's theorem; weight.
\end{keywords}
\begin{sloppy}
\begin{classification}
Primary 43A99; Secondary 22D12, 43A15, 43A32, 46H05, 46J10.
\end{classification}
\end{sloppy}
\section*{Introduction}
A \emph{weight} on a locally compact group $G$ is a measurable, locally integrable function $\omega \!: G \to [1,\infty)$ such that
\begin{equation} \label{submult}
  \omega(xy) \leq \omega(x) \omega(y) \qquad (x,y \in G).
\end{equation}
The corresponding \emph{Beurling algebra} (\cite[Definition 3.7.2]{RSt}) is defined as
\[
  L^1(G,\omega) := \{ f \in L^1(G) : \omega f \in L^1(G) \}.
\]
It is a subalgebra of $L^1(G)$ and a Banach algebra in its own right with respect to the norm $\| \cdot \|_\omega$ given by $\| f \|_\omega := \| \omega f \|_1$ for $f \in L^1(G, \omega)$. There is no loss of generality if we suppose that $\omega$ is continuous (\cite[Theorem 3.7.5]{RSt}). Beurling algebras have been objects of study in abstract harmonic analysis for a long time, especially for abelian $G$ (see \cite{Kan} and \cite{RSt}, for instance).
\par 
If $G$ is abelian with dual group $\hat{G}$, then the Fourier transform is an isometric isomorphism between $L^1(G)$ and the Fourier algebra $A(\hat{G})$ of $\hat{G}$. Consequently, if $\omega$ is any weight on $G$, then $L^1(G,\omega)$ is isomorphic to a subalgebra of $A(\hat{G})$. In \cite{Eym}, P.\ Eymard defined the Fourier algebra $A(G)$ for general, not necessarily abelian, locally compact groups $G$. This brings up the natural question if there is a way to define certain subalgebras of $A(G)$, which, for abelian $G$, correspond to the Beurling algebras on $L^1(\hat{G})$.
\par 
In \cite{LS}, H.-H.\ Lee and E.\ Samei introduced the notion of a Beurling--Fourier algebra. If $G$ is a locally compact group and $\omega \!: G \to [1,\infty)$ is a weight, then multiplication with $\omega$ defines a closed, densely defined operator on $L^2(G)$, which is bounded if and only if $\omega$ is bounded, i.e., $L^1(G,\omega)$ is trivial. Consequently, Lee and Samei define what they call a \emph{weight on the dual of $G$} as a closed, densely defined operator on $L^2(G)$ affiliated with the group von Neumann algebra $\VN(G)$. The resulting theory of Beurling--Fourier algebras is particularly tractable for what Lee and Samei call central weights on the duals of compact groups. Independently, these weights and their corresponding Beurling--Fourier algebras were also introduced and investigated by J.\ Ludwig, L.\ Turowska, and the third-named author (\cite{LST}).
\par 
The approach in \cite{LST} is restricted to compact groups, and both in \cite{LS} and \cite{LST}, it is unclear if the given definitions of a Beurling--Fourier algebra can be extended beyond the $L^2$-context to define weighted variants of the Fig\`a-Talamanca--Herz algebras (see \cite{Eym2}, \cite{FT}, \cite{Her1}, \cite{Her2}, and \cite{Spe}). In the present note, we propose a different approach to Beurling--Fourier algebras with the following features:
\begin{itemize}
\item if $G$ is a locally compact abelian group with dual group $\hat{G}$, then the Beurling--Fourier algebras correspond---via the Fourier transform---to the Beurling algebras on $\hat{G}$;
\item at least for compact $G$, our approach subsumes the one from \cite{LS} (and thus of \cite{LST});
\item the definitions extend effortlessly from the $L^2$-framework to a general $L^p$-context with $p \in (1,\infty)$, which enables us to define \emph{Beurling--Fig\`a-Talamanca--Herz algebras}.
\end{itemize}
\par 
The key idea is to not attempt to define a ``dual'' notion of weight, but rather that of the inverse of a weight. This approach enables us to define Beurling--Fourier algebras without any reference to the theory of von Neumann algebras, on which \cite{LS} relies heavily, so that it can be adapted to an $L^p$-context. 
\par 
For the resulting Beurling--Fig\`a-Talamanca--Herz algebras, we obtain an extension of the Leptin--Herz theorem, which characterizes the amenable locally compact groups through the existence of bounded approximate identities in their Fig\`a-Talamanca--Herz algebras: a locally compact group is amenable if and only if one---or, equivalently, every---of its Beurling--Fig\`a-Talamanca--Herz algebras has a bounded approximate identity. 
\section{Beurling algebras through inverses of weights}
We shall suppose throughout that all weights are continuous: by \cite[Theorem 3.7.5]{RSt}, this is no limitation if one is only interested in the corresponding Beurling algebras.
\par
If $G$ is a locally compact group and $\omega \!: G \to [1,\infty)$ is a weight, then $\omega$ is bounded if and only if $L^1(G,\omega) = L^1(G)$ with an equivalent norm, i.e., unless $L^1(G,\omega)$ is trivial, the multiplication operator induces by $\omega$ on $L^2(G)$ is unbounded. The inverse of $\omega$---with respect to pointwise multiplication---, however, is bounded on $G$, i.e., the corresponding multiplication operator on $L^2(G)$ is bounded and thus lies in the multiplier algebra of $\mathcal{C}_0(G)$, the $\cstar$-algebra of all continuous functions on $G$ vanishing at infinity (represented on $L^2(G)$ as multiplication operators).
\par
For a locally compact group $G$, we denote by $\mathcal{C}_b(G)$ the $\cstar$-algebra of all bounded continuous functions on $G$. We note the following:
\begin{proposition} \label{inverseweightprop}
Let $G$ be a locally compact group. Then the following are equivalent for non-negative $\alpha \in \mathcal{C}_b(G)$ with $\| \alpha \|_\infty \leq 1$:
\begin{items}
\item there is a weight $\omega \!: G \to [1,\infty)$ such that $\alpha = \omega^{-1}$;
\item \begin{alphitems}
\item the map
\begin{equation}
  \mathcal{C}_0(G) \to \mathcal{C}_0(G), \quad f \mapsto \alpha f
\end{equation}
has dense range;
\item there is $\Omega \in L^\infty(G \times G)$ with $\| \Omega \|_\infty \leq 1$ such that
\begin{equation} \label{bigO}
  \alpha(x) \alpha(y) = \alpha(xy) \Omega(x,y) \qquad (x,y \in G).
\end{equation}
\end{alphitems}
\end{items}
Moreover, if $\omega$ is as in \emph{(i)}, then
\[
  L^1(G,\omega) = \{ \alpha f : f \in L^1(G) \}
\]
and
\[
  \| \alpha f \|_\omega = \| f \|_1 \qquad (f \in L^1(G)).
\]
\end{proposition}
\begin{proof}
(i) $\Longrightarrow$ (ii): Set
\[
  \Omega(x,y) := \frac{\omega(xy)}{\omega(x) \omega(y)} \qquad (x,y \in G).
\]
From (\ref{submult}), it is immediate that $\Omega \in \mathcal{C}_b(G \times G) \subset L^\infty(G \times G)$ with $\| \Omega \|_\infty \leq 1$, and by definition, (\ref{bigO}) holds. Hence, (a) is satisfied. To see that (b) holds, note that $\{ \alpha f : f \in \mathcal{C}_0(G) \}$ is a self-adjoint subalgebra of $\mathcal{C}_0(G)$ that strongly separates the points of $G$; it is therefore dense in $\mathcal{C}_0(G)$ by the Stone--Weierstra{\ss} theorem.
\par 
(ii) $\Longrightarrow$ (i): From (b), it is immediate that $\alpha (x) \neq 0$ for all $x \in G$. Hence, we can define $\omega := \alpha^{-1}$. As $\| \alpha \|_\infty \leq 1$, it is clear that $\omega(G) \subset [1,\infty)$. From (a), it follows that $\omega$ satisfies (\ref{submult}).
\par 
The ``moreover'' part is obvious.
\end{proof}
\par 
The bottom line of Proposition \ref{inverseweightprop} is that Beurling algebras can be defined without any reference to a weight---a possibly unbounded continuous function---, but rather through the inverses of weights, which are bounded continuous functions, i.e., multipliers of $\mathcal{C}_0(G)$.
\par
To adapt the notion of the inverse of a weight to context of Fourier algebras, we introduce the notion of a Hopf--von Neumann algebra (see \cite{ES}). As is customary, we write $\bar{\tensor}$ for the tensor product of von Neumann algebras.
\begin{definition}
A \emph{Hopf--von Neumann algebra} is a pair $(M,\Gamma)$ where $M$ is a von Neumann algebra and $\Gamma \!: M \to M \bar{\tensor} M$ is a \emph{co-multiplication}, i.e., a normal, faithful, unital $^\ast$-homomorphism such that
\[
  (\Gamma \tensor \id) \circ \Gamma = (\id \tensor \Gamma) \circ \Gamma.
\]
\end{definition}
\par 
Whenever $(M,\Gamma)$ is a Hopf--von Neumann algebra, the unique predual $M_\ast$ of $M$ becomes a Banach algebra with respect to the product $\ast$ defined via
\begin{equation} \label{predualprod}
  \langle f \ast g, x \rangle := \langle f \tensor g, \Gamma x \rangle \qquad (f, g \in M_\ast, \, x \in M).
\end{equation}
If $M_\ast$ is equipped with its canonical operator space structure (see \cite{ER} for background on the theory of operator spaces), then (\ref{predualprod}) defines not only a contractive, but completely contractive bilinear map, thus turning $M_\ast$ into a completely contractive Banach algebra (see \cite[p.\ 308]{ER}).
\begin{example}
Let $G$ be a locally compact group, and $M = L^\infty(G)$---so that $M_\ast = L^1(G)$ and $L^\infty(G) \bar{\tensor} L^\infty(G) \cong L^\infty(G \times G)$---, and define $\Gamma \!: L^\infty(G) \to L^\infty(G \times G)$ through
\[
  (\Gamma \phi)(x,y) := \phi(xy) \qquad (\phi \in L^\infty(G), \, x,y \in G).
\]
It is easy to check that the product on $L^1(G)$ in the sense of (\ref{predualprod}) is just the ordinary convolution product on $L^1(G)$.
\end{example}
\par 
The first part of Proposition \ref{inverseweightprop} can thus be rephrased as:
\begin{corollary} \label{inverseweightcor}
Let $G$ be a locally compact group. Then the following are equivalent for non-negative $\alpha \in \mathcal{C}_b(G)$ with $\| \alpha \|_\infty \leq 1$:
\begin{items}
\item there is a weight $\omega \!: G \to [1,\infty)$ such that $\alpha = \omega^{-1}$;
\item \begin{alphitems}
\item the map
\begin{equation}
  \mathcal{C}_0(G) \to \mathcal{C}_0(G), \quad f \mapsto \alpha f
\end{equation}
has dense range;
\item there is $\Omega \in L^\infty(G \times G)$ with $\| \Omega \|_\infty \leq 1$ such that
\[
  \alpha \tensor \alpha = (\Gamma \alpha)\Omega.
\]
\end{alphitems}
\end{items}
\end{corollary}
\section{Weight inverses and Beurling--Fourier algebras}
Let $G$ be a locally compact group, and let $\lambda \!: G \to \mathcal{B}(L^2(G))$ be the left regular representation of $G$ on $L^2(G)$, i.e.,
\[
  (\lambda(x) \xi)(y) := \xi(x^{-1}y) \qquad (\xi \in L^2(G), \, x,y \in G).
\]
Through integration, $\lambda$ ``extends'' to a $^\ast$-representation of the group algebra $L^1(G)$; we use the symbol $\lambda$ for it as well. We define
\[
  C^\ast_r(G) := \overline{\lambda(L^1(G))}^{\| \cdot \|} \qquad\text{and}\qquad \VN(G) := \overline{\lambda(L^1(G))}^{\text{weak$^\ast$}},
\]
the \emph{reduced group $\cstar$-algebra} and the \emph{group von Neumann algebra} of $G$, respectively. The \emph{Fourier algebra} $A(G)$ of $G$ is the predual of $\VN(G)$ (see \cite{Eym}).
\par 
We introduce a co-multiplication $\hat{\Gamma} \!: \VN(G) \to \VN(G) \bar{\tensor} \VN(G) \cong \VN(G \times G)$, thus turning $A(G)$ into a completely contractive Banach algebra. To this end, define $W \in \mathcal{B}(L^2(G \times G))$ via
\[
  (W \boldsymbol{\xi})(x,y) := \boldsymbol{\xi}(x,xy) \qquad (\boldsymbol{\xi} \in L^2(G \times G), \, x,y \in G). 
\]
Then
\[
  \hat{\Gamma} \!: \mathcal{B}(L^2(G)) \to \mathcal{B}(L^2(G \times G)), \quad T \mapsto W^{-1}(T \tensor 1)W
\]
is a co-multiplication, satisfying
\[
  \hat{\Gamma} \lambda(x) = \lambda(x) \tensor \lambda(x) \qquad (x \in G);
\]
it follows that $\hat{\Gamma} \VN(G) \subset \VN(G \times G)$. Let the product on $A(G)$ induced by $\hat{\Gamma}$ be denoted by $\hat{\ast}$. Given $f,g \in A(G)$ and $x \in G$, we have
\[
  \langle f \hat{\ast} g, \lambda(x) \rangle = \langle f \tensor g, \hat{\Gamma} \lambda(x) \rangle = \langle f \tensor g, \lambda(x) \tensor \lambda(x) \rangle = f(x) g(x),
\]
i.e., $\hat{\ast}$ is pointwise multiplication.
\par 
Whenever $M$ is a von Neumann algebra, its predual $M_\ast$ is an $M$-bimodule in a canonical manner:
\[
  \langle x, y f \rangle := \langle xy,f \rangle = \langle y, fx \rangle \qquad (f \in M_\ast, \, x, y \in M).
\]
Also, if $A$ is a $\cstar$-algebra, we write $\mathcal{M}(A)$ for its \emph{multiplier algebra}.
\par 
With an eye on Corollary \ref{inverseweightcor}, we define:
\begin{definition} \label{weightinversedef}
Let $G$ be a locally compact group $G$. A \emph{weight inverse} is an element $\omega^{-1}$ of $\mathcal{M}(C^\ast_r(G))$ with $\| \omega^{-1} \| \leq 1$  such that the following are satisfied:
\begin{alphitems}
\item the maps
\begin{equation} \label{denserange1}
  C^\ast_r(G) \to C^\ast_r(G), \quad x \mapsto x \omega^{-1}
\end{equation}
and
\begin{equation} \label{denserange2}
  C^\ast_r(G) \to C^\ast_r(G), \quad x \mapsto \omega^{-1} x
\end{equation}
have dense range;
\item there is $\Omega \in \VN(G \times G)$ with $\| \Omega \| \leq 1$ such that 
\[
  \omega^{-1} \tensor \omega^{-1} = (\hat{\Gamma} \omega^{-1})\Omega;
\]
\end{alphitems}
The corresponding \emph{Beurling--Fourier algebra} is defined as
\[
  A(G,\omega) := \{ \omega^{-1} f : f \in A(G) \}.
\]
\end{definition}
\begin{remarks}
\item We have not defined what $\omega$ is: the $\omega$ in $A(G,\omega)$ is thus purely symbolic. However, a simple Hahn--Banach argument shows that $\omega^{-1} \!: L^2(G) \to L^2(G)$ is injective with dense range (as is $(\omega^{-1})^\ast$). We can thus define $\omega \!: \omega^{-1}L^2(G) \to L^2(G)$ as the inverse of $\omega^{-1} \!: L^2(G) \to \omega^{-1} L^2(G)$. It is immediate (\cite[Proposition II.6.2]{Yos}) that $\omega$ is closable and thus extends to a  closed (necessarily densely defined) operator on $L^2(G)$.
\item If $\omega^{-1}$ is self-adjoint, then it is sufficient that one of (\ref{denserange1}) or (\ref{denserange2}) have dense range.
\end{remarks}
\par
At the first glance, it may seem bewildering that we do not require weight inverses to be \emph{positive} elements of $\mathcal{M}(C^\ast_r(G))$. The reason for this is that we are interested in later extending Definition \ref{weightinversedef} to an $L^p$-context for general $p \in (1,\infty)$, where there is no suitable notion of positivity available. Still, not requiring in Definition \ref{weightinversedef} that $\omega^{-1}$ be positive, does not yield any more Beurling--Fourier algebras, as the next proposition shows:
\begin{proposition} \label{posweight}
Let $G$ be a locally compact group, and let $\omega^{-1} \in \mathcal{M}(C^\ast_r(G))$ be a weight inverse. Then $|(\omega^{-1})^\ast | \in \mathcal{M}(C^\ast_r(G))$ is also a weight inverse such that the corresponding Beurling--Fourier algebra coincides with $A(G,\omega)$.
\end{proposition}
\begin{proof}
Due to Definition \ref{weightinversedef}(a), the sets $\{ x \omega^{-1} : x \in C^\ast_r(G) \}$ and $\{ \omega^{-1} x : x \in C^\ast_r(G) \}$ are dense in $C^\ast_r(G)$, as are $\{ x (\omega^{-1})^\ast : x \in C^\ast_r(G) \}$ and $\{ (\omega^{-1})^\ast x : x \in C^\ast_r(G) \}$.
\par 
Let $(\omega^{-1})^\ast = u | (\omega^{-1})^\ast |$ be the polar decomposition of $(\omega^{-1})^\ast$. Then 
\begin{multline*}
  \{ x | (\omega^{-1})^\ast | : x \in C^\ast_r(G) \} \supset \\ \{ x | (\omega^{-1})^\ast | | (\omega^{-1})^\ast | : x \in C^\ast_r(G) \} =
  \{ (x \omega^{-1}) (\omega^{-1})^\ast  : x \in C^\ast_r(G) \} 
\end{multline*}
is dense in $C^\ast_r(G)$, as is---by an analogous argument---$\{ | (\omega^{-1})^\ast | x: x \in C^\ast_r(G) \}$, i.e.,
\[
  C^\ast_r(G) \to C^\ast_r(G), \quad x \mapsto x | (\omega^{-1})^\ast |
\]
and
\[
  C^\ast_r(G) \to C^\ast_r(G), \quad x \mapsto | (\omega^{-1})^\ast | x
\]
each have dense range.
\par 
As we remarked after Definition \ref{weightinversedef}, $(\omega^{-1})^\ast$ is injective with dense range. Consequently, the partial isometry $u$ must be unitary; note also that $u \in \VN(G)$ (\cite[Proposition II.3.14]{Tak}). Let $\Omega$ be as in Definition \ref{weightinversedef}(b). Then we have
\[
  | (\omega^{-1})^\ast | \tensor | (\omega^{-1})^\ast | = (\omega^{-1} \tensor \omega^{-1})(u \tensor u) = (\hat{\Gamma} \omega^{-1})\Omega(u \tensor u) =
  (\hat{\Gamma} | (\omega^{-1})^\ast | ) (\hat{\Gamma} u) \Omega(u \tensor u).
\]
As $\| (\hat{\Gamma} u) \Omega(u \tensor u) \| = \| \Omega \| \leq 1$, it follows that $| (\omega^{-1})^\ast |$ satisfies Definition \ref{weightinversedef}(b) with $(\hat{\Gamma} u) \Omega(u \tensor u)$ \emph{en lieu} of $\Omega$.
\par 
Finally, note that
\[
  A(G,\omega) = \{ \omega^{-1} f : f \in A(G) \} = \{ \omega^{-1} u f : f \in A(G) \}  
              = \{ | (\omega^{-1})^\ast | f : f \in A(G) \},
\]
so that the Beurling--Fourier algebras corresponding to $\omega^{-1}$ and $|(\omega^{-1})^\ast|$ coincide.
\end{proof}
\par
For abelian groups and positive weight inverses, the Beurling--Fourier algebras in the sense of Definition \ref{weightinversedef} are in perfect duality with the classical Beurling algebras, as we shall now see.
\par 
If $G$ is a locally compact abelian group with dual group $\hat{G}$, we always suppose that Haar measures on $G$ and $\hat{G}$ are scaled such that the Fourier inversion formula (\cite[1.5.1, Theorem]{Rud}) holds. In this case, there is a unique unitary $\mathcal{P} \!: L^2(G) \to L^2(\hat{G})$---the \emph{Plancherel transform}---that coincides with the Fourier transform $\mathcal{F} \!: L^1(G) \to A(\hat{G})$ on $L^1(G) \cap L^2(\hat{G})$. By $\hat{\mathcal{F}}$ and $\hat{\mathcal{P}}$, we denote the Fourier and Plancherel transforms, respectively, arising from $\hat{G}$. Also, if $G$ is any locally compact group and if $\phi \in L^\infty(G)$, we denote the corresponding multiplication operator on $L^2(G)$ by $M_\phi$ (slightly abusing notation, we shall often write $\phi$ instead of $M_\phi$). Finally, if $G$ is a locally compact group, and $\phi \!: G \to \comps$ is any function, we define functions $\bar{\phi}$, $\check{\phi}$, and $\tilde{\phi}$ on $G$ by letting
\[
  \bar{\phi}(x) := \overline{\phi(x)}, \quad \check{\phi}(x) := \phi(x^{-1}), \quad\text{and}\quad \tilde{\phi}(x) := \overline{\phi}(x^{-1}) \qquad (x \in G).
\]
\par
The following lemma is known by all likelihood, but for lack of a suitable reference, we give a proof:
\begin{lemma} \label{planchl}
Let $G$ be a locally compact abelian group with dual group $\hat{G}$. Then we have:
\begin{equation} \label{intertw}
  \mathcal{P}^\ast \lambda(f) \mathcal{P} = M_{(\hat{\mathcal{F}} f)^\vee} \qquad (f \in L^1(\hat{G})).
\end{equation}
\end{lemma}
\begin{proof}
Let $f, \xi \in L^1(G) \cap L^2(G)$. Then
\[
  \mathcal{P}(\lambda(f) \xi) = \mathcal{F}(f \ast \xi) = (\mathcal{F}f) (\mathcal{P}\xi)
\]
holds. It follows that
\begin{equation} \label{intertw0}
  \mathcal{P} \lambda(f) \mathcal{P}^\ast = M_{\mathcal{F}f} \qquad (f \in L^1(G)).
\end{equation}
\par 
Let $V \in \mathcal{B}(L^2(G))$ be the unitary operator given by $V\xi := \check{\xi}$ for $\xi \in L^2(G)$. It is routinely checked that $\mathcal{P}^\ast = V \hat{\mathcal{P}}$. Replacing the r\^oles of $G$ and $\hat{G}$, we obtain from (\ref{intertw0}) that
\[
  \mathcal{P}^\ast \lambda(f) \mathcal{P} = V  \hat{\mathcal{P}} \lambda(f) \hat{\mathcal{P}}^\ast V = V M_{\hat{\mathcal{F}} f} V = M_{(\hat{\mathcal{F}} f)^\vee} \qquad (f \in L^1(\hat{G})),
\]
which proves (\ref{intertw}).
\end{proof}
\begin{lemma} \label{adjFourier}
Let $G$ be a locally compact abelian group with dual group $\hat{G}$, and let $\mathcal{F} \!: L^1(G) \to A(\hat{G})$ be the Fourier transform. Then $\mathcal{F}^\ast \!: \VN(\hat{G}) \to L^\infty(G)$ is a $^\ast$-isomorphism that maps $C^\ast_r(\hat{G})$ onto $\mathcal{C}_0(G)$ and satisfies
\begin{equation} \label{comultcomp}
  (\mathcal{F}^\ast \tensor \mathcal{F}^\ast) \circ \hat{\Gamma} = \Gamma \circ \mathcal{F}^\ast.
\end{equation}
\end{lemma}
\begin{proof}
Since $\mathcal{F}$ is an isomorphism of Banach algebras and since the multiplication in $L^1(G)$ and $A(\hat{G})$, respectively, arises from $\Gamma$ and $\hat{\Gamma}$, respectively, it is clear that (\ref{comultcomp}) holds. 
\par
To tell the Hilbert space inner product of $L^2(\hat{G})$ apart from a Banach space duality $\langle \cdot, \cdot \rangle$, we write $\langle \cdot | \cdot \rangle$. Let $f, g \in L^1(G)$, and let $\xi, \eta \in L^2(G)$ be such that $g = \xi \bar{\eta}$; we have:
\[
  \begin{split}
  \langle g, \mathcal{F}^\ast(\lambda(f)) \rangle & = \langle \xi \bar{\eta}, \mathcal{F}^\ast(\lambda(f)) \rangle \\
  & = \langle \mathcal{F} (\xi\bar{\eta}), \lambda(f) \rangle \\
  & = \left\langle \mathcal{P}\xi \ast \widetilde{\mathcal{P} \eta}, \lambda(f) \right\rangle \\
  & = \langle \lambda(f) \mathcal{P}\xi | \mathcal{P}\eta \rangle \\
  & = \langle \mathcal{P}^\ast \lambda(f) \mathcal{P}\xi | \eta \rangle \\
  & = \langle M_{(\hat{\mathcal{F}}f)^\vee} \xi | \eta \rangle, \qquad\text{by Lemma \ref{planchl}}, \\
  & = \langle g, (\hat{\mathcal{F}} f)^\vee \rangle \\
  \end{split}
\]
It follows that
\begin{equation} \label{intertw+}
  \mathcal{F}^\ast (\lambda(f)) = \mathcal{P}^\ast \lambda(f) \mathcal{P} = (\hat{\mathcal{F}}f)^\vee.
\end{equation}
From the first equality, it is immediate that $\mathcal{F}^\ast$ is a $^\ast$-homomorphism. The second equality shows that $\mathcal{F}^\ast$ maps the algebra $\lambda(L^1(\hat{G}))$, which is dense in $C^\ast_r(\hat{G})$ onto a dense subalgebra of $\mathcal{C}_0(G)$.
\end{proof}
\begin{proposition} \label{dualprop}
Let $G$ be a locally compact abelian group with dual group $\hat{G}$, and let $\mathcal{F} \!: L^1(G) \to A(\hat{G})$ be the Fourier transform. Then: 
\begin{items}
\item a continuous function $\omega \!: G \to [1,\infty)$ is a weight if and only if $(\hat{\omega})^{-1} := (\mathcal{F}^\ast)^{-1}(\omega^{-1})$ is a positive weight inverse, in which case $\mathcal{F}(L^1(G,\omega)) = A(\hat{G},\hat{\omega})$;
\item if $\omega^{-1} \in \mathcal{M}(C^\ast_r(\hat{G}))$ is a positive weight inverse, then $(\mathcal{F}^\ast \omega^{-1})^{-1}$ is a weight on $G$.
\end{items}
\end{proposition}
\begin{proof}
As $\mathcal{F}^\ast$ is a $^\ast$-isomorphism mapping $C^\ast_r(\hat{G})$ onto $\mathcal{C}_0(G)$, it is clear that $\mathcal{F}^\ast$ and its inverse respect positivity and map the closed unit balls of $\mathcal{M}(C^\ast_r(\hat{G}))$ and $\mathcal{C}_b(G)$ onto each other.
\par 
Let $\omega$ be a weight on $G$. Then $\alpha := (\hat{\omega})^{-1}$ is a non-negative function in the unit ball of $\mathcal{C}_b(G)$ satisfying Corollary \ref{inverseweightcor}(ii)(a). From Lemma \ref{adjFourier}, we conclude that $(\hat{\omega})^{-1} := (\mathcal{F}^\ast)^{-1}(\alpha)$ satisfies Definition \ref{weightinversedef}, i.e., is a weight inverse.
\par
Conversely, if $\omega^{-1} \in \mathcal{M}(C^\ast_r(\hat{G}))$ is a positive weight inverse, then $\alpha := \mathcal{F}^\ast \omega^{-1}$ is a non-negative function in the unit ball of $\mathcal{C}_b(G)$ satisfying Corollary \ref{inverseweightcor}(ii), so that $(\mathcal{F}^\ast \omega^{-1})^{-1}$ is a weight by that corollary.
\par 
Let $\omega \!: G \to [1,\infty)$ be a weight, and let $(\hat{\omega})^{-1}$ be defined as in (i). To see that $\mathcal{F}(L^1(G,\omega)) = A(\hat{G},\hat{\omega})$, first note that
\[
  (\mathcal{F}^\ast)^{-1} \phi = \mathcal{P} M_\phi \mathcal{P}^\ast \qquad (\phi \in L^\infty(G))
\]
by (\ref{intertw+}). Let $f \in L^1(G)$. Choose $\xi, \eta \in L^2(G)$ such that $f = \xi \bar{\eta}$. Then we have:
\[
  \mathcal{F}(\omega^{-1}f) = \mathcal{F}(\omega^{-1}\xi \bar{\eta}) = \mathcal{P}(M_{\omega^{-1}} \xi) \ast \widetilde{\mathcal{P}\eta}
  = ((\mathcal{P}M_{\omega^{-1}}\mathcal{P}^\ast)\mathcal{P}\xi) \ast \widetilde{\mathcal{P}\eta} = (\mathcal{F}^\ast)^{-1} (\omega^{-1})(\mathcal{F}f).
\]
This completes the proof.
\end{proof}
\par 
We now give examples for weight inverses in the sense of Definition \ref{weightinversedef}:
\begin{examples}
\item Let $G$ be a locally compact abelian group with dual group $\hat{G}$, and let $\mathcal{F} \!: L^1(G) \to A(\hat{G})$ denote the Fourier transform. By Proposition \ref{dualprop}(i), $(\hat{\omega})^{-1} := (\mathcal{F}^\ast )^{-1}(\omega^{-1})$ is a weight inverse for every weight $\omega \!: G \to [1,\infty)$, and $\mathcal{F}(L^1(G,\omega)) = A(\hat{G},\hat{\omega})$ holds. By Proposition \ref{dualprop}(ii), every weight inverse in $\mathcal{M}(C^\ast_r(\hat{G}))$ arises in this fashion.
\item In \cite{LS}, H.\ H.\ Lee and E.\ Samei define Beurling--Fourier algebras using an explicit definition of a weight on the dual of a locally compact group $G$ (\cite[Definition 2.4]{LS}). A weight in their sense is a closed, densely defined, positive operator $\omega$ on $L^2(G)$ affiliated with $\VN(G)$ satisfying various properties. In particular, they require:
\begin{items}
\item $\omega$ has a bounded inverse $\omega^{-1} \in \VN(G)$;
\item $(\hat{\Gamma} \omega)(\omega^{-1} \tensor \omega^{-1}) \leq 1$ (for the definition of $\hat{\Gamma}\omega$, see \cite{LS});
\item $\{ x \omega^{-1} : x \in VN(G) \}$ is weak$^\ast$ dense in $\VN(G)$.
\end{items}
By multiplying $\omega$, if necessary, with a positive scalar, there is also no loss of generality to suppose that $\| \omega^{-1} \| \leq 1$. 
\par 
Suppose that $G$ is compact, so that
\begin{equation}\label{compactcase}
  C_r^\ast(G)^{\ast\ast} = \VN(G) = \mathcal{M}(C^\ast_r(G)),
\end{equation}
and let $\omega$ be a weight on the dual of $G$ in the sense of \cite[Definition 2.4]{LS}. We claim that $\omega^{-1}$ is a weight inverse in the sense of Definition \ref{weightinversedef}. First of all, note that $\omega^{-1} \in \mathcal{M}(C^\ast_r(G))$ by (\ref{compactcase}). Since the weak$^\ast$ topology of $\VN(G)$ restricted to $C^\ast_r(G)$ is the weak topology, we obtain that $\{ x \omega^{-1} : x \in C^\ast_r(G) \}$ is norm dense in $C^\ast_r(G)$; as $\omega^{-1}$ is positive, $\{ \omega^{-1} x : x \in C^\ast_r(G) \}$ is also norm dense in $C^\ast_r(G)$. Finally, set $\Omega := (\hat{\Gamma} \omega)(\omega^{-1} \tensor \omega^{-1})$, so that
\[
  \omega^{-1} \tensor \omega^{-1} =(\hat{\Gamma}\omega^{-1}) (\hat{\Gamma} \omega)(\omega^{-1} \tensor \omega^{-1}) = (\hat{\Gamma}\omega^{-1})\Omega.
\]
\par
It follows that the central weights discussed in \cite[Subsection 2.2]{LS} as well as the weights introduced in \cite[Section 3]{LST}, all yield weight inverses in the sense of Definition \ref{weightinversedef}. 
\end{examples}
\par 
So far, we have only labeled the Beurling--Fourier ``algebras'' algebras without showing that they are indeed algebras.
\par
For the next theorem note that, if $G$ is a locally compact group and $\omega^{-1} \in \mathcal{M}(C^\ast_r(G))$ is a weight inverse. Then (\ref{denserange1}) has a dense range, so that its adjoint is injective, as is the restriction
\begin{equation} \label{rats}
  A(G) \to A(G), \quad f \mapsto \omega^{-1} f
\end{equation}
to $A(G)$.
\begin{theorem} \label{BFthm}
Let $G$ be a locally compact group, and let $\omega^{-1} \in \mathcal{M}(C^\ast_r(G))$ be a weight inverse. Then $A(G,\omega)$ is a dense subalgebra of $A(G)$. Moreover, if $A(G,\omega)$ is equipped with the unique operator space structure turning the bijection
\[
  A(G) \to A(G,\omega), \quad f \mapsto \omega^{-1} f
\]
into a complete isometry, it is a completely contractive Banach algebra.
\end{theorem}
\par 
We refrain from giving a proof here because we will prove a more general result in the context of Fig\`a-Talamanca--Herz algebras (see Theorem \ref{BFTHthm} below).
\section{Beurling--Fig\`a-Talamanca--Herz algebras}
Let $G$ once again be a locally compact group, let $p \in (1,\infty)$, and let $\lambda_p \!: G \to \mathcal{B}(L^p(G))$ be the left regular representation of $G$ on $L^p(G)$, meaning
\[
  (\lambda_p(x) \xi)(y) := \xi(x^{-1}y) \qquad (\xi \in L^p(G), \, x,y \in G); 
\]
we also write $\lambda_p$ for the representation of $L^1(G)$ on $L^p(G)$ obtained through integration. We define
\[
  \PF_p(G) := \overline{\lambda_p(L^1(G))}^{\| \cdot \|} \qquad\text{and}\qquad \PM_p(G) := \overline{\lambda_p(L^1(G))}^{\text{weak$^\ast$}},
\]
the \emph{$p$-pseudofunctions} and the \emph{$p$-pseudomeasures} on $G$, respectively; we also define
\[
  \mathcal{M}(\PF_p(G)) := \{ x \in \PM_p(G) : \text{$x\PF_p(G) \subset \PF_p(G)$ and $\PF_p(G)x \subset \PF_p(G)$} \}.
\]
The $p$-pseudomeasures form a weak$^\ast$ closed subspace of the dual Banach space $\mathcal{B}(L^p(G))$ and thus have a canonical predual, the \emph{Fig\`a-Talamanca--Herz algebra} $A_p(G)$. 
\par 
For what follows, we require the theory of $p$-operator spaces, which is outlined in \cite{Daw}, for instance.
\par 
There are a $p$-completely contractive, weak$^\ast$ continuous map $\hat{\Gamma}_p \!: \PM_p(G) \to \PM_p(G \times G)$ with
\[
  \hat{\Gamma}_p \lambda_p(x) = \lambda_p(x) \tensor \lambda_p(x) \qquad (x \in G)
\]
as well as a canonical, weak$^\ast$ continuous, $p$-complete contraction $\theta \!: \PM_p(G \times G) \to (A_p(G) \Tensor_p A_p(G))^\ast$ such that the preadjoint $(\theta \hat{\Gamma}_p)_\ast \!: A_p(G) \Tensor_p A_p(G) \to A_p(G)$ is pointwise multiplication (here, $\Tensor_p$ stands for the projective tensor product of $p$-operator spaces; see \cite{Daw} for the definition).
\par
As $A(G)$ is a completely contractive $\VN(G)$-bimodule, $A_p(G)$ is a $p$-completely contractive $\PM_p(G)$-bimodule. We can thus extend Definition  \ref{weightinversedef}:
\begin{definition} \label{pweightinversedef}
Let $G$ be a locally compact group $G$, and let $p \in (1,\infty)$. A \emph{weight inverse} is an element $\omega^{-1}$ of $\mathcal{M}(\PF_p(G))$-with $\| \omega^{-1} \| \leq 1$  such that the following are satisfied:
\begin{alphitems}
\item the maps
\begin{equation} \label{denserange3}
  \PF_p(G) \to \PF_p(G), \quad x \mapsto x \omega^{-1}
\end{equation}
and
\begin{equation} \label{denserange4}
  \PF_p(G) \to \PF_p(G), \quad x \mapsto \omega^{-1} x
\end{equation}
have dense range;
\item there is $\Omega \in \PM_p(G \times G)$ with $\| \Omega \| \leq 1$ such that 
\[
  \omega^{-1} \tensor \omega^{-1} = (\hat{\Gamma}_p \omega^{-1})\Omega.
\]
\end{alphitems}
The corresponding \emph{Beurling--Fig\`a-Talamanca--Herz algebra} is defined as
\[
  A_p(G,\omega) := \{ \omega^{-1} f : f \in A_p(G) \}.
\]
\end{definition}
\par
We have a canonical extension of Theorem \ref{BFthm} to Beurling--Fig\`a-Talamanca--Herz algebras:
\begin{theorem} \label{BFTHthm}
Let $G$ be a locally compact group, let $p \in (1,\infty)$, and let $\omega^{-1} \in \mathcal{M}(\PF_p(G))$ be a weight inverse. Then $A_p(G,\omega)$ is a dense subalgebra of $A_p(G)$. Moreover, if $A_p(G,\omega)$ is equipped with the unique $p$-operator space structure turning the bijection
\begin{equation} \label{injective}
  A_p(G) \to A_p(G,\omega), \quad f \mapsto \omega^{-1} f
\end{equation}
into a complete isometry, it is a $p$-completely contractive Banach algebra. 
\end{theorem}
\begin{proof}
To show that $A_p(G,\omega)$ is dense in $A_p(G)$, let $x \in \PM_p(G)$ be such that $\langle f, x \rangle = 0$ for $f \in A_p(G,\omega)$, i.e., $\langle \omega^{-1} f, x \rangle = \langle f, x \omega^{-1} \rangle = 0$ for $f \in A_p(G)$. It follows that $x \omega^{-1} = 0$. As (\ref{denserange4}) has dense range in $\PF_p(G)$, the set $\{ \omega^{-1} y : y \in \PM_p(G) \}$ is weak$^\ast$ dense in $\PM_p(G)$, so that $x \PM_p(G) = \{ 0 \}$. Since $\PM_p(G)$ is unital, we conclude that $x = 0$, so that $A_p(G,\omega)$ is dense in $A_p(G)$ by the Hahn--Banach theorem.
\par 
To see that $A_p(G,\omega)$ is multiplicatively closed, let $f,g \in A_p(G)$, let $x \in \PM_p(G)$, and note that
\[
  \begin{split}
  \langle (\omega^{-1} f)(\omega^{-1} g),x \rangle & = \langle (\omega^{-1} \tensor \omega^{-1})(f \tensor g), \theta \hat{\Gamma}_p x \rangle \\
  & = \langle (\omega^{-1} \tensor \omega^{-1}) \theta_\ast(f \tensor g), \hat{\Gamma}_p x \rangle \\
  & = \langle (\hat{\Gamma}_p \omega^{-1}) \Omega \theta_\ast(f \tensor g), \hat{\Gamma}_p x \rangle \\
  & = \langle \Omega \theta_\ast (f \tensor g), \hat{\Gamma}_p(x \omega^{-1}) \rangle \\
  & = \langle (\hat{\Gamma}_p)_\ast(\Omega \theta_\ast(f \tensor g)), x \omega^{-1} \rangle \\
  & = \langle  \omega^{-1} (\hat{\Gamma}_p)_\ast(\Omega \theta_\ast(f \tensor g)), x \rangle,
  \end{split}
\]
i.e.,
\begin{equation} \label{BFTHprod}
  (\omega^{-1} f)(\omega^{-1} g) = \omega^{-1} (\hat{\Gamma}_p)_\ast(\Omega \theta_\ast(f \tensor g)) \in A_p(G,\omega).
\end{equation}
Hence, $A_p(G,\omega)$ is a subalgebra of $A_p(G)$.
\par
Since $\theta_\ast$, $(\hat{\Gamma}_p)_\ast$, and
\[
  A_p(G \times G) \to A_p(G \times G), \quad F \mapsto \Omega F
\]
are $p$-complete contractions, so is their composition, it follows from (\ref{BFTHprod}) that $A_p(G,\omega)$ is a $p$-completely contractive Banach algebra.
\end{proof}
\par 
If $G$ is a locally compact group and $\omega$ is a weight on $G$, then $L^1(G,\omega) = L^1(G)$ with equivalent norms if and only if $\omega$ is bounded, which is trivially satisfied for compact $G$. In view of the duality between $L^1$- and Fourier algebras, one should expect a similar result for Beurling--Fourier algebras which should always be true on discrete groups. Indeed, this holds even for Beurling--Fig\`a-Talamanca--Herz algebras:
\begin{proposition}
Let $G$ be a locally compact group, and let $\omega^{-1} \in \mathcal{M}(\PF_p(G))$ be a weight inverse. Then the following are equivalent:
\begin{items}
\item the inclusion map from $A_p(G,\omega)$ into $A_p(G)$ is surjective;
\item the inclusion map from $A_p(G,\omega)$ into $A_p(G)$ is surjective and has a $p$-completely bounded inverse;
\item $\omega^{-1}$ is left invertible in $\PM_p(G)$.
\end{items}
If $G$ is discrete, then $\omega^{-1}$ is automatically invertible in $\mathcal{M}(\PF_p(G))$, so that \emph{(i)}, \emph{(ii)}, and \emph{(iii)} hold.
\end{proposition}
\begin{proof}
(iii) $\Longrightarrow$ (ii) $\Longrightarrow$ (i) hold trivially.
\par
(i) $\Longrightarrow$ (iii): The composition of (\ref{injective}) with the canonical inclusion of $A_p(G,\omega)$ into $A_p(G)$ is 
\[
  A_p(G) \to A_p(G), \quad f \mapsto \omega^{-1} f.
\]
If this map is bijective, then so is
\[
  \PM_p(G) \to \PM_p(G), \quad x \mapsto x \omega^{-1}.
\]
As $\PM_p(G)$ is unital, $\omega^{-1}$ must be left invertible in $\PM_p(G)$.
\par
If $G$ is discrete then, $\PF_p(G)$ is unital, so that $\mathcal{M}(\PF_p(G)) = \PF_p(G)$. From the density of the ranges of (\ref{denserange3}) and (\ref{denserange4}) in $\PF_p(G)$ it is clear that the left ideals $\{ x\omega^{-1} : x\in \PF_p(G) \}$ and $\{ \omega^{-1} x : x \in \PF_p(G) \}$ are both dense in $\PF_p(G)$ and thus, by basic Banach algebra theory, all of $\PF_p(G)$. It follows that $\omega^{-1}$ is invertible in $\PF_p(G)$.
\end{proof}
\begin{remark}
Suppose that $p =2$ and that $\omega^{-1} \in \mathcal{M}(C^\ast_r(G))$ is a weight inverse such that (\ref{rats}) is surjective (and thus, automatically, bijective). As $(\omega^{-1})^\ast = u | (\omega^{-1})^\ast|$ with $u \in \VN(G)$ unitary---see the proof of Proposition \ref{posweight}---, it follows that
\[
  A(G) \mapsto A(G), \quad f \mapsto |( \omega^{-1})^\ast| f
\]
is also bijective. Hence, $|(\omega^{-1})^\ast|$ is left invertible in $\VN(G)$ and, being self-adjoint, actually invertible. This entails that $(\omega^{-1})^\ast$ is invertible in $\VN(G)$ and thus in $\mathcal{M}(C^\ast_r(G))$. In the $p=2$ situation, we thus have the equivalence of:
\begin{items}
\item the inclusion map from $A(G,\omega)$ into $A(G)$ is surjective;
\item the inclusion map from $A(G,\omega)$ into $A(G)$ is surjective and has a completely bounded inverse;
\item $\omega^{-1}$ is invertible in $\mathcal{M}(C^\ast_r(G))$.
\end{items}
\end{remark}
\section{A weighted Leptin--Herz theorem}
It is well known that, for any locally compact group $G$ and any weight $\omega \!: G \to [1,\infty)$, the Beurling algebra $L^1(G,\omega)$ has a bounded approximate identity (\cite[Proposition 3.7.7]{RSt}). On the other hand, H.\ Leptin proved in \cite{Lep} that a locally compact group $G$ is amenable if and only if $A(G)$ has a bounded approximate identity. This result was subsequently extended in to Fig\`a-Talamanca--Herz algebras in \cite{Her2} by C.\ Herz, who claimed this extension to be folklore. 
\par 
In this section, we prove a weighted version of the Leptin--Herz theorem: a locally compact group $G$ is amenable if and only if, for all $p \in (1,\infty)$ and all weight inverses $\omega^{-1} \in \mathcal{M}(\PF_p(G))$, the algebra $A_p(G,\omega)$ has a bounded approximate identity.
\par 
As $A_p(G,\omega)$ is dense in $A_p(G)$ with the inclusion being ($p$-completely) contractive, any bounded approximate identity for $A_p(G,\omega)$ is automatically an approximate identity for $A_p(G)$, thus forcing $G$ to be amenable. If one tries to adapt the proof in the unweighted case---via F{\o}lner type conditions---, difficulties show up immediately: in general, the functions in $A_p(G,\omega)$ with compact support need not be dense in $A_p(G,\omega)$. We thus pursue a different route, which is inspired by the theory of Kac algebras (see \cite{ES}).
\begin{definition} \label{Ppdef}
Let $G$ be a locally compact group, and let $p \in [1,\infty)$. We call a net $(\xi_\alpha )_\alpha$ of non-negative norm one functions in $L^p(G)$ a \emph{$(P_p)$-net} if
\[
  \sup_{x \in K} \| \lambda_p(x) \xi_\alpha - \xi_\alpha \|_p \to 0.
\]
for a compact $K \subset G$.
\end{definition}
\begin{remark}
The choice of terminology in Definition \ref{Ppdef} is, of course, due to property $(P_p)$ introduced by H.\ Reiter (see \cite[Definition 8.3.1]{RSt}). A locally compact group $G$ is amenable if and only if it has Property $(P_p)$ for one---and, equivalently, for all---$p \in [1,\infty)$ (\cite[Theorem 6.14]{Pie}), i.e., there is a $(P_p)$-net in $L^p(G)$.
\end{remark}
\par 
For any $p \in (1,\infty)$, the $(P_p)$-nets are defined in terms of an asymptotic invariance property. For the proof of our weighted Leptin--Herz theorem, we require three more such properties, which we formulate as three lemmas.
\begin{lemma} \label{Pplem}
Let $G$ be an amenable locally compact group, and let $p \in (1,\infty)$. Then:
\begin{items} 
\item the augmentation character $1 \in L^\infty(G)$ on $L^1(G)$ extends uniquely to a multiplicative linear functional on $\mathcal{M}(\PF_p(G))$;
\item for any $(P_p)$-net $( \xi_\alpha )_\alpha$ in $L^p(G)$, we have
\begin{equation} \label{Ppeq}
  \| x \xi_\alpha - \langle x,1 \rangle \xi_\alpha \|_p \to 0 \qquad (x \in \mathcal{M}(\PF_p(G)).
\end{equation}
\end{items}
\end{lemma}
\begin{proof}
It follows from \cite[Theorem 5]{Cow} that $1$ extends (necessarily uniquely) to $\PF_p(G)$ as a (necessarily multiplicative) bounded linear functional. Fix $a \in \PF_p(G)$ with $\langle a, 1 \rangle =1$, and define 
\[
  \phi \!: \mathcal{M}(\PF_p(G)) \to \comps, \quad x \mapsto \langle xa, 1 \rangle.
\]
Clearly, $\phi$ is a continuous functional extending $1$. To see that $\phi$ is multiplicative, let $( e_\alpha )_\alpha$ be a bounded approximate identity for $L^1(G)$, so that $(\lambda_p(e_\alpha) )_\alpha$ is a bounded approximate identity for $\PF_p(G)$. We obtain for $x,y \in \mathcal{M}(\PF_p(G))$:
\begin{multline*}
  \langle xy, \phi \rangle = \langle xya, 1 \rangle = \lim_\alpha \langle x \lambda(e_\alpha) ya, 1 \rangle = \lim_\alpha \langle x \lambda(e_\alpha) , 1 \rangle \langle y a , 1 \rangle \\ = \lim_\alpha \langle x \lambda(e_\alpha) a , 1 \rangle \langle ya,1 \rangle = \langle xa, 1 \rangle \langle ya,1 \rangle = \langle x, \phi \rangle \langle y, \phi \rangle.
\end{multline*}
It is obvious that $\phi$ is the only multiplicative extension of $1$ from $\PF_p(G)$ to $\mathcal{M}(\PF_p(G))$ (for the sake of simplicity, we will also denote this extension by $1$). This proves (i).
\par 
For the proof of (ii), first note that (\ref{Ppeq}) holds for $x \in \lambda_p(L^1(G))$: this is due to the fact the the functions with compact support are dense in $L^1(G)$. Due to the norm density of $\lambda(L^1(G))$ in $\PF_p(G)$, we obtain (\ref{Ppeq}) for $x \in \PF_p(G)$ as well. Finally, let $x \in \mathcal{M}(\PF_p(G))$ be arbitrary. Fix $a \in \PF_p(G)$ with $\langle a, 1 \rangle = 1$, so that $\| a \xi_\alpha -  \xi_\alpha \|_p \to 0$ and thus 
\[
  \| x \xi_\alpha - xa \xi_\alpha \|_p \to 0.
\]
As $\langle xa,1 \rangle = \langle x, 1 \rangle$, we obtain
\[
  \| x \xi_\alpha - \langle x,1 \rangle \xi_\alpha \|_p \leq \| x \xi_\alpha - xa \xi_\alpha \|_p + \| xa \xi - \langle xa, 1 \rangle \xi_\alpha \|_p \to 0.
\]
This completes the proof.
\end{proof}
\par 
Let $p \in (1,\infty)$ be arbitrary. As in the case $p=2$, we define $W_p \in \mathcal{B}(L^p(G \times G))$ by letting
\[
  (W_p \boldsymbol{\xi})(x,y) := \boldsymbol{\xi}(x,xy) \qquad (\boldsymbol{\xi} \in L^p(G \times G), \, x,y \in G),
\] 
Like in that case, we have
\begin{equation} \label{comultdef}
  \hat{\Gamma}_p x = W^{-1}_p(x \tensor 1) W_p \qquad (x \in \PM_p(G \times G)).
\end{equation}
Observe also that, if $q \in (1,\infty)$ is dual to $p$, i.e., $\frac{1}{p} + \frac{1}{q} = 1$, then 
\begin{equation} \label{adjoints}
  W_p^\ast = W^{-1}_q \qquad\text{and}\qquad (W_p^{-1})^\ast = W_q.
\end{equation}
\begin{lemma} \label{Wpinv}
Let $G$ be a locally compact group, let $p \in (1,\infty)$, and let $( \xi_\alpha )_\alpha$ be a $(P_p)$-net in $L^p(G)$. Then we have
\begin{equation} \label{asy1}
  \| W_p(\eta \tensor \xi_\alpha) - \eta \tensor \xi_\alpha \|_p \to 0 \qquad (\eta \in L^p(G))
\end{equation}
and
\begin{equation} \label{asy2}
  \| W_p^{-1} (\eta \tensor \xi_\alpha) - \eta \tensor \xi_\alpha \| \to 0 \qquad (\eta \in L^p(G)).
\end{equation}
\end{lemma}
\begin{proof}
If $\eta$ has compact support, (\ref{asy1}) is immediate from Definition \ref{Ppdef}; the general case follows by the usual density argument. Clearly, (\ref{asy2}) follows from (\ref{asy1}).
\end{proof}
\begin{lemma} \label{asylem}
Let $G$ be a locally compact group, let $p,q \in (1,\infty)$ be dual to each other, let $\omega^{-1} \in \mathcal{M}(\PF_p(G))$, let $\Omega \in \PM_p(G \times G)$ be as in \emph{Definition \ref{pweightinversedef}(iii)}, and let $( \xi_\alpha )_\alpha$ be a $(P_q)$-net in $L^q(G)$. Then we have
\begin{equation} \label{asy3}
  \| \Omega^\ast(\eta \tensor \xi_\alpha) - \eta \tensor (\omega^{-1})^\ast \xi_\alpha \|_q \to 0 \qquad (\eta \in L^p(G)).
\end{equation}
\end{lemma}
\begin{proof}
By Definition \ref{pweightinversedef}(iii) and (\ref{comultdef}), we have
\[
  \omega^{-1} \tensor \omega^{-1} = (\hat{\Gamma} \omega^{-1}) \Omega = W_p^{-1}(\omega^{-1} \tensor 1) W_p \Omega.
\]
Through taking adjoints---taking (\ref{adjoints}) into account---, we obtain 
\begin{equation} \label{omegaeq}
  (\omega^{-1})^\ast \tensor (\omega^{-1} )^\ast = \Omega^\ast W_q^{-1}((\omega^{-1})^\ast \tensor 1) W_q.
\end{equation}
As $( \xi_\alpha )_\alpha$ is a $(P_q)$-net, we obtain from Lemma \ref{Wpinv} that
\[
  \| \Omega^\ast W_q^{-1}((\omega^{-1})^\ast \tensor 1) W_q(\eta \tensor \xi_\alpha) - \Omega^\ast((\omega^{-1})^\ast \eta \tensor \xi_\alpha) \|_q \to 0.
\]
In view of (\ref{omegaeq}), this yields (\ref{asy3}) in the case where $\eta \in (\omega^{-1})^\ast L^q(G)$; the general case follows from the fact that $(\omega^{-1})^\ast L^q(G)$ is dense in $L^q(G)$.
\end{proof}
\par 
For our next result, the technical heart of our argument, recall the notion of a \emph{weak approximate identity} of a Banach algebra $A$: this a net $( e_\alpha)_\alpha$ in $A$ such that 
\[
  a e_\alpha \to a \quad\text{and}\quad e_\alpha a \to a \qquad (a \in A)
\]
in the weak topology of $A$ (see, for instance, \cite[Definition 11.3]{BD}). 
\begin{proposition} \label{weakbai}
Let $G$ be a locally compact group, let $( \xi_\alpha )_{\alpha \in \mathcal{A}}$ be a $(P_1)$-net in $L^1(G)$, let $p,q \in (1,\infty)$ be dual to each other, and let the net $( e_\alpha )_{\alpha \in \mathbb{A}}$ in $A_p(G)$ be defined by
\[
  e_\alpha(x) := \left\langle \lambda_p(x) \xi_\alpha^\frac{1}{p}, \xi_\alpha^\frac{1}{q}\right\rangle \qquad (x \in G, \, \alpha \in \mathbb{A}).
\]
Then, if $\omega^{-1} \in \mathcal{M}(\PF_p(G))$ is a weight inverse, the net $\left( \langle \omega^{-1}, 1 \rangle^{-1} \omega^{-1} e_\alpha \right)_{\alpha \in \mathbb{A}}$ in $A_p(G,\omega)$ is a weak approximate identity for $A_p(G, \omega)$.
\end{proposition}
\begin{proof}
It is clear that $( \langle \omega^{-1}, 1 \rangle^{-1} \omega^{-1} e_\alpha )_{\alpha \in \mathbb{A}}$ is bounded in $A_p(G,\omega)$. Also note that
\[
  (\omega^{-1} e_\alpha)(x) = \left\langle \lambda_p(x) \omega^{-1} \xi_\alpha^\frac{1}{p}, \xi_\alpha^\frac{1}{q}\right\rangle \qquad (x \in G, \, \alpha \in \mathbb{A}).
\]
\par 
Let $f \in A_p(G)$. Without loss of generality, suppose that there are $\eta \in L^p(G)$ and $\zeta \in L^q(G)$ with $\langle \omega^{-1} \eta, \zeta \rangle =1$ such that $f(x) = \langle \lambda_p(x) \eta, \zeta \rangle$ for $x \in G$; this means that $(\omega^{-1}f)(x) = \langle \lambda_p(x) \omega^{-1} \eta, \zeta \rangle$ for $x \in G$.
\par 
There is a canonical complete isomorphism $\kappa \!: \PM_p(G) \to A_p(G,\omega)^\ast$, given by
\[
  \langle \omega^{-1} f, \kappa(x) \rangle = \langle f,x \rangle \qquad (x \in \PM_p(G)). 
\]
Fix $x \in \PM_p(G)$. 
\par 
From the proof of Theorem \ref{BFTHthm}, we see that for $\alpha \in \mathbb{A}$:
\begin{equation} \label{ouch}
  \begin{split}
  \lefteqn{\langle (\omega^{-1} f) \langle\omega^{-1},1\rangle^{-1} \omega^{-1} e_\alpha, \kappa(x) \rangle} & \\ 
  & = \langle\omega^{-1},1\rangle^{-1} \langle \omega^{-1} (\hat{\Gamma}_p)_\ast(\Omega \theta_\ast(f \tensor e_\alpha)), \kappa(x) \rangle \\ & = \langle\omega^{-1},1\rangle^{-1} \langle \Omega \theta_\ast(f \tensor e_\alpha), \hat{\Gamma}_p x \rangle \\
  & = \langle\omega^{-1},1\rangle^{-1} \langle \theta_\ast(f \tensor e_\alpha), (\hat{\Gamma}_p x)\Omega \rangle \\
  & = \langle\omega^{-1},1\rangle^{-1} \langle \theta_\ast(f \tensor e_\alpha), W_p^{-1}(x \tensor 1)W_p \Omega \rangle \\
  & = \langle\omega^{-1},1\rangle^{-1} \left\langle \zeta \tensor \xi_\alpha^\frac{1}{q}, W_p^{-1}(x \tensor 1)W_p \Omega\left(\eta \tensor \xi_\alpha^\frac{1}{p} \right) \right\rangle \\
  & = \langle\omega^{-1},1\rangle^{-1} \left\langle \Omega^\ast W_q^{-1}(x^\ast \tensor 1)W_q\left(\zeta \tensor \xi_\alpha^\frac{1}{q}\right), \left(\eta \tensor \xi_\alpha^\frac{1}{p} \right) \right\rangle.
  \end{split}
\end{equation}
As $( \xi_\alpha )_{\alpha \in \mathbb{A}}$ is a $(P_1)$ net, $\left( \xi_\alpha^\frac{1}{q} \right)_{\alpha \in \mathbb{A}}$ is a $(P_q)$-net. From Lemmas \ref{Wpinv} and \ref{asylem}, we conclude that
\[
  \left\| \Omega^\ast W_q^{-1}(x^\ast \tensor 1)W_q\left(\zeta \tensor \xi_\alpha^\frac{1}{q}\right) - x^\ast \zeta \tensor (\omega^{-1})^\ast \xi_\alpha \right\|_q \to 0
\]
and thus
\begin{multline*}
  \left\langle \zeta \tensor \xi_\alpha^\frac{1}{q}, W_p^{-1}(x \tensor 1)W_p \Omega\left(\eta \tensor \xi_\alpha^\frac{1}{p} \right) \right\rangle - \left\langle \zeta \tensor \xi_\alpha^\frac{1}{q}, x\eta \tensor \omega^{-1} \xi_\alpha^\frac{1}{p}\right\rangle \\ = \left\langle \Omega^\ast W_q^{-1}(x^\ast \tensor 1)W_q\left(\zeta \tensor \xi_\alpha^\frac{1}{q}\right) - x^\ast \zeta \tensor (\omega^{-1})^\ast \xi_\alpha, \eta \tensor \xi_\alpha^\frac{1}{p} \right\rangle \to 0.
\end{multline*}
Together with (\ref{ouch}), this yields 
\begin{equation} \label{moreouch}
  \lim_\alpha \langle (\omega^{-1} f) \langle\omega^{-1},1\rangle^{-1} \omega^{-1} e_\alpha, \kappa(x) \rangle - \langle\omega^{-1},1\rangle^{-1}  \left\langle \zeta \tensor \xi_\alpha^\frac{1}{q}, x\eta \tensor \omega^{-1} \xi_\alpha^\frac{1}{p}\right\rangle = 0.
\end{equation}
\par 
On the other hand, as $( \xi_\alpha )_{\alpha \in \mathbb{A}}$ is a $(P_1)$ net, $\left( \xi_\alpha^\frac{1}{p} \right)_{\alpha \in \mathbb{A}}$ is a $(P_p)$-net, so that 
\[
  \left\| \omega^{-1}\xi_\alpha^\frac{1}{p} - \langle \omega^{-1},1 \rangle \xi_\alpha^\frac{1}{p} \right\|_p \to 0
\]
by Lemma \ref{Pplem}(ii) and thus
\begin{equation} \label{evenmoreouch}
  \left\langle \zeta \tensor \xi_\alpha^\frac{1}{q}, x\eta \tensor \omega^{-1} \xi_\alpha^\frac{1}{p}\right\rangle = \langle f, x \rangle \left\langle \xi_\alpha^\frac{1}{q}, \omega^{-1} \xi_\alpha^\frac{1}{p} \right\rangle \to \langle f, x \rangle \langle \omega^{-1}, 1 \rangle =  \langle \omega^{-1}f, \kappa(x) \rangle \langle \omega^{-1}, 1 \rangle.
\end{equation}
\par
Combined, (\ref{moreouch}) and (\ref{evenmoreouch}) yield
\[
  \lim_\alpha \langle (\omega^{-1} f) \langle\omega^{-1},1\rangle^{-1} \omega^{-1} e_\alpha - f, \kappa(x) \rangle = 0.
\]
As $x \in \PM_p(G)$ was arbitrary, this completes the proof.
\end{proof}
\par 
Summing everything up, we obtain:
\begin{theorem}
The following are equivalent for a locally compact group $G$:
\begin{items}
\item $G$ is amenable;
\item for every $p \in (1,\infty)$ and for every weight inverse $\omega^{-1} \in \mathcal{M}(\PF_p(G))$, the Beurling--Fig\`a-Talamanca--Herz algebra $A_p(G,\omega)$ has a bounded approximate identity;
\item there are $p \in (1,\infty)$ and a weight inverse $\omega^{-1} \in \mathcal{M}(\PF_p(G))$ such that the Beurling--Fig\`a-Talamanca--Herz algebra $A_p(G,\omega)$ has a bounded approximate identity
\end{items}
\end{theorem}
\begin{proof}
(i) $\Longrightarrow$ (ii): Let $ p \in (1,\infty)$, and let $\omega^{-1} \in \mathcal{M}(\PF_p(G))$ be a weight inverse. As $G$ is amenable, it has Reiter's property $(P_1)$ (\cite[Proposition 6.12]{Pie}), i.e., there is a $(P_1)$-net in $L^1(G)$. By Proposition \ref{weakbai}, this means that $A_p(G,\omega)$ has a weak bounded approximate identity. By a standard Banach algebra result (see \cite[Proposition 11.4]{BD}, for example), this means that $A_p(G,\omega)$ already has a bounded approximate identity.
\par
(ii) $\Longrightarrow$ (iii) is trivial.
\par 
(iii) $\Longrightarrow$ (i): As we remarked at the beginning of this section, the existence of a bounded approximate identity for $A_p(G,\omega)$ already implies the existence of one for $A_p(G)$. By the unweighted Leptin--Herz theorem (\cite[Theorem 10.4]{Pie}), this means that $G$ is amenable.
\end{proof}
\renewcommand{\baselinestretch}{1.0}
\renewcommand{\baselinestretch}{1.2}
\begin{tabbing}
\textit{Second author's address}: \= Department of Mathematical and Statistical Sciences \kill 
\textit{First author's address}: \> Department of Mathematics \\
\> Faculty of Science \\
\> Istanbul University \\ 
\> Istanbul \\
\> Turkey \\[\medskipamount]
\textit{E-mail}: \> \texttt{oztops@istanbul.edu.tr} \\[\bigskipamount]
\textit{Second author's address}: \> Department of Mathematical and Statistical Sciences \\
\> University of Alberta \\ 
\> Edmonton, Alberta \\
\> Canada T6G 2G1 \\[\medskipamount]
\textit{E-mail}: \> \texttt{vrunde@ualberta.ca} \\[\bigskipamount]
\textit{Third author's address}: \> Department of Pure Mathematics \\
\> University of Waterloo \\
\> Waterloo, ON \\
\> Canada N2L 3G1 \\[\medskipamount]
\textit{E-mail}: \> \texttt{nspronk@math.uwaterloo.ca} 
\end{tabbing}   
\dated        

\begin{thebibliography}{L--S--T} \begin{small}
%
\bibitem[B--D]{BD} \textsc{F.\ F.\ Bonsall} and \textsc{J.\ Duncan}, \textit{Complete Normed Algebras}. Ergebnisse der Mathematik und ihrer Grenzgebiete \textbf{80}, Springer Verlag, 1973.
%
\bibitem[Cow]{Cow} \textsc{M.\ Cowling}, An application of Littlewood--Paley theory in harmonic analysis. \textit{Math.\ Ann.}\ \textbf{241} (1979), 83--96.
%
\bibitem[Daw]{Daw} \textsc{M.\ Daws}, $p$-operator spaces and Fig\`a-Talamanca–-Herz algebras. \textit{J.\ Operator Theory} \textbf{63} (2010), 47-–83.
%
\bibitem[E--R]{ER} \textsc{E.\ G.\ Effros} and \textsc{Z.-J.\ Ruan}, \textit{Operator Spaces}. London Mathematical Society Monographs (New Series) \textbf{23}, Clarendon Press, 2000.
%
\bibitem[E--S]{ES} \textsc{M.\ Enock} and \textsc{J.-M.\ Schwartz}, \textit{Kac Algebras and Duality of Locally Compact Groups}. Springer Verlag, 1992.
%
\bibitem[Eym 1]{Eym} \textsc{P.\ Eymard}, L'alg\`ebre de Fourier d'un groupe localement compact. \textit{Bull.\ Soc.\ Math.\ France} \textbf{92} (1964), 
181--236. 
%
\bibitem[Eym 2]{Eym2} \textsc{P.\ Eymard}, Alg\`ebres $A_p$ et convoluteurs de $L^p$. In: \textit{S\'eminaire Bourbaki, vol.\ 1969/70, Expos\'es 364--381}, Lecture Notes in Mathematics \textbf{180}, Springer Verlag, 1971. 
%
\bibitem[F-T]{FT} \textsc{A.\ Fig\`a-Talamanca}, Translation invariant operators in $L^p$. \textit{Duke Math.\ J.}\ \textbf{32} (1965) 495--501.
%
\bibitem[Her 1]{Her1} \textsc{C.\ Herz}, The theory of $p$-spaces with an application to convolution operators. \textit{Trans.\ Amer.\ Math.\ Soc.}\ \textbf{154} (1971), 69--82.
%
\bibitem[Her 2]{Her2} \textsc{C.\ Herz}, Harmonic synthesis for subgroups. \textit{Ann.\ Inst.\ Fourier (Grenoble)} \textbf{23} (1973), 91--123.
%
\bibitem[Kan]{Kan} \textsc{E.\ Kaniuth}, \textit{A Course in Commutative Banach Algebras}. Graduate Texts in Mathematics \textbf{246}, Springer Verlag, 2009.
%
\bibitem[Lep]{Lep} \textsc{H.\ Leptin}, Sur l'alg\`ebre de Fourier d'un groupe localement compact. \textit{C.\ R.\ Acad.\ Sci.\ Paris\/}, S\'er.\ A \textbf{266} (1968), 1180--1182.
%
\bibitem[L--S]{LS} \textsc{H.\ H.\ Lee} and \textsc{E.\ Samei}, Beurling--Fourier algebras, operator amenability, and Arens regularity. \textit{J.\ Funct.\ Anal.}\ \textbf{262} (2012), 167--209.
%
\bibitem[L--S--T]{LST} \textsc{J.\ Ludwig}, \textsc{N.\ Spronk}, and \textsc{L.\ Turowska}, Beurling--Fourier algebras on compact groups: spectral theory. \textit{J.\ Funct.\ Anal.}\ \textbf{262} (2012), 463--499,
%
\bibitem[Pie]{Pie} \textsc{J.-P.\ Pier}, \textit{Amenable Locally Compact Groups}. Wiley-Interscience, 1984.
%
\bibitem[R--St]{RSt} \textsc{H.\ Reiter} and \textsc{J.\ D.\ Stegeman}, \emph{Classical Harmonic Analysis and Locally Compact Groups}. London Mathematical Society Monographs (New Series) \textbf{22}, Clarendon Press, 2000.
%
\bibitem[Rud]{Rud} \textsc{W.\ Rudin}, \textit{Fourier Analysis on Groups}. Wiley Classics Library, John Wiley \& Sons, 1990.
%
\bibitem[Spe]{Spe} \textsc{R.\ Spector}, Sur la structure locale des groupes ab\'eliens localement compacts. \textit{Bull.\ Soc.\ Math.\ France Suppl.\ M\'em.}\ \textbf{24} (1970).
%
\bibitem[Tak]{Tak} \textsc{M.\ Takesaki}, \textit{Theory of Operator Algebras}, I. Encyclopedia of Mathematical Sciences \textbf{124},
Springer Verlag, 2003.
%
\bibitem[Yos]{Yos} \textsc{K.\ Yosida}, \textit{Functional Analysis}. Grundlehren der mathematischen Wissenschaften \textbf{123}, Springer Verlag, 1980.
%
\end{small} \end{thebibliography}
\end{document}